\documentclass[14pt]{article}
\usepackage{mathpazo}
\usepackage{amsmath}
\usepackage{amsfonts}
\usepackage{array,color}
\usepackage{abstract}
\usepackage[bottom]{footmisc}
\usepackage[left=1.15in, right=1.15in, bottom=0.60in, includefoot]{geometry}
\usepackage{graphicx}
\usepackage{indentfirst}
\usepackage{amsthm}
\usepackage{mathrsfs}
\usepackage{makeidx}
\usepackage{url}
\usepackage{dsfont}
\usepackage{latexsym}
\usepackage{amsmath}
\usepackage{amssymb}
\usepackage{color}
\usepackage{mathtools}
\usepackage{hyperref}
\usepackage{tikz}
\usepackage{MnSymbol}
\numberwithin{equation}{subsection}

\newcommand{\G}{\Gamma}

\newcommand{\sg}{\sigma}
\newcommand{\mc}{\mathbb{C}}
\newcommand{\ca}{\curvearrowright}

\newcommand{\emm}{\mathcal{M}}

\newcommand{\enn}{\mathcal N}
\newcommand{\euu}{\mathcal{U}}

\newcommand{\El}{\mathcal{L}}

\newcommand{\rar}{\rightarrow}
\newcommand{\La}{\Lambda}

\newcommand{\tp}{\bar{\otimes}}

\newcommand{\Aa}{\mathcal{A}}
\newcommand{\bee}{\mathcal{B}}

\newcommand{\mr}{\mathcal{R}}
\newcommand{\ms}{\mathcal{S}}

\newcommand{\pee}{\mathcal{P}}

\newcommand{\Diag}{\operatorname{Diag}}
\newcommand{\Sp}{\operatorname{Sp}}
\newcommand{\GL}{\operatorname{GL}}

\begin{document}
	\newtheorem{Lemma}{Lemma}
	\theoremstyle{plain}
	\newtheorem{theorem}{Theorem~}[section]
	\newtheorem{main}{Main Theorem~}
	\newtheorem{lemma}[theorem]{Lemma~}
	\newtheorem{assumption}[theorem]{Assumption~}
	\newtheorem{proposition}[theorem]{Proposition~}
	\newtheorem{corollary}[theorem]{Corollary~}
	\newtheorem{definition}[theorem]{Definition~}
	\newtheorem{defi}[theorem]{Definition~}
	\newtheorem{notation}[theorem]{Notation~}
	\newtheorem{example}[theorem]{Example~}
	\newtheorem*{remark}{Remark~}
	\newtheorem{cor}{Corollary~}
	\newtheorem*{question}{Question}
	\newtheorem{claim}{Claim}
	\newtheorem*{conjecture}{Conjecture~}
	\newtheorem*{fact}{Fact~}
	\newtheorem*{thma}{Theorem A}
	\newtheorem*{thmb}{Theorem B}
	\renewcommand{\proofname}{\bf Proof}
	\newcommand{\email}{Email: }

	\title {New examples of Property (T) factors with trivial fundamental group and unique prime factorization \vskip 0.02 in
		\fontsize{11}{1}	\textit {Dedicated to the loving memory of Prof.\ Vaughan Jones} }
	
	\author{ Sayan Das}
	\date{}
	\maketitle
	%{\centering Dedicated to the loving memory of Prof.\ Vaughan Jones. \par}
	\begin{abstract}
		In this paper we provide new examples of property (T) group factors with trivial fundamental group thereby providing more evidence towards Popa's conjecture on triviality of fundamental groups for property (T) group factors (page 9 \cite{Po13}; see also Problem 2, page 551 in Connes' book \cite{Co94}). Our groups arise as direct product of groups either in Class $\mathscr S$ or Class $\mathscr V$, as introduced in \cite{CDHK20}. We establish a unique prime factorization result for these product groups, thereby providing more evidence towards Popa's conjecture on prime decomposition of property (T) factors (page 9 \cite{Po13}).
	\end{abstract}
	
	\section{Introduction}
	Fundamental group is one of the very first isomorphism invariants of von Neumann algebras introduced by Murray and von Neumann in \cite{MvN43}. In \cite{MvN43}, they showed that the fundamental group of any McDuff factor is $\mathbb R_{+}$, but were unable to find examples of $\rm II_1$ factors with fundamental group different from $\mathbb R_{+}$. Over the years, the calculation of fundamental groups of specific type $\rm II_1$ factors was firmly established as one of the most challenging problems in the theory of von Neumann algebras. In fact, the first examples of $\rm II_1$ factors with fundamental group different from $\mathbb R_{+}$ was only found in 1980 by A. Connes \cite{Co80}. Connes showed that $\mathcal F(\El(G))$ is at most countable, whenever $G$ is an i.c.c. property (T) group. This beautiful result prompted Connes to formulate his famous rigidity conjecture about $W^*$-superrigidity of i.c.c. property (T) groups in \cite{Co82} (see also \cite[Problem 1, page 551]{Co94}). However, Connes' result still did not provide any explicit calculation of fundamental group of any of the aforementioned $\rm II_1$ factors. In particular, the problem of finding $\rm II_1$ factors with trivial fundamental group remained elusive.
	\vskip 0.01 in
Indeed, despite Connes' stunning result, it would take two more decades and Sorin Popa to finally obtain explicit examples of $\rm II_1$ factors with trivial fundamental group, \cite{Po01}. The breakthrough by Popa came via his deformotation/rigidity theory. In the next two decades, Popa's deformation/rigidity theory proved indispensable in tackling hitherto impossible problems in von Neumann algebras, such as providing first examples of $\rm II_1$ factors with trivial outer automorphism group (\cite{IPP05}), first examples of $W^*$-superrigid groups (\cite{IPV10}), first examples of virtually $W^*$ superrigid group actions (\cite{Pet10, PV10}) and many more. We refer the reader to the surveys \cite{Po07, Va10icm, Io18icm} for recent progress in von Neumann algebras using deformation/rigidity theory.
\vskip 0.02 in
 In spite of the aforementioned breakthrough results, no explicit calculations of fundamental groups of property (T) factors was known until very recently. In \cite{CDHK20}, I. Chifan, C. Houdayer, K. Khan and the author provided the first examples of property (T) group factors with trivial fundamental group. In fact, in \cite{CDHK20}, it was shown that $\mathcal F(\El(\G))=\{1\}$ for $\G \in \mathscr S \cup \mathscr V $, where $\mathscr S$ and $\mathscr V$ both contain infinitely many i.c.c. property (T) groups. We refer the reader to \cite[Section 3]{CDHK20} for the definition of the classes of groups $\mathscr S$ and $\mathscr V$. The aforementioned result provided the first supporting evidence towards the strong form of Connes' rigidity conjecture \cite{Jo00,Po07}. The main goal of this paper is to provide new examples of property (T) group factors with trivial fundamental group, thereby advancing\cite[Problem 2, page 551]{Co94} and providing new examples satisfying the last conjecture on page 9 in \cite{Po13}. The main theorem in this paper is the following
 
 \begin{thma}
 	Let $\G_1, \G_2 \in \mathscr S \cup \mathscr V$, and let $\G=\G_1 \times \G_2$. Then $\mathcal F(\El(\G))=\{1\}$.
 \end{thma}

The proof of Theorem A proceeds via establishing unique prime factorization (UPF) results for the group factors $\El(\G)$, as in Theorem A. More precisely, we prove the following 
\begin{thmb}
	Let $\G_1, \G_2 \in \mathscr S \cup \mathscr V$, and let $\G=\G_1 \times \G_2$. Let $\ms_1$, $\ms_2$ be $\rm II_1$ factors such that $\El(\G)= \ms_1 \tp \ms_2$. Then there exists $t>0$, and $u \in \euu(\El(\G))$, such that $u\ms_1^tu^*= \El(\G_1)$ and $u\ms_2^{1/t}u^*= \El(\G_2)$ (or vice versa).	
	\end{thmb}
The study of unique prime factorization and its use in calculation of fundamental groups of $\rm II_1$ factors was pioneered by Ozawa and Popa in \cite{OP03}. In fact, \cite[Proposition 12]{OP03} remains one of the main tools for establishing UPF results till date. Unique prime factorization was further studied in \cite{DHI16, CdSS17}. The relation of calculation of fundamental groups and UPF was also explored in the beautiful paper by Isono \cite{Is20}. The aforementioned papers provide valuable insight into the study of unique prime factorization results, and the author owes an intellectual debt to these papers. 
\vskip 0.02 in
Finally we mention that we prove two other UPF results. In Theorem ~\ref{resfinupf} we establish a unique prime factorization result for direct products of i.c.c., property (T) groups, that are hyperbolic relative to infinite residually finite subgroups. The other UPF result, Theorem ~\ref{upf}, establishes unique prime factorization of tensor product of a prime full factor with a group von Neumann algebra of an icc,hyperbolic group.
\vskip 0.01 in
As a final remark we note that since the group factors in Theorem B, and Theorem ~\ref{resfinupf} have property (T), these Theorems provide further evidence towards Popa's conjecture on prime decomposition of property (T) factors (second last problem on page 9 in \cite{Po13}; see also Problem S.4 in Peterson's problem list \cite{PProb}).
	\section{Preliminaries}
	\subsection{Notations and Terminology}
	
	\vskip 0.05in	
	Throughout this document all von Neumann algebras are denoted by calligraphic letters e.g.\ $\mathcal A$, $\mathcal B$, $\mathcal M$, $\mathcal N$, etc. Given a von Neumann algebra $\mathcal M$ we will denote by $\mathscr U(\mathcal M)$ its unitary group, by $\mathscr P(\mathcal M)$ the set of all its nonzero projections and by $(\mathcal M)_1$ its unit ball.  Given a unital inclusion $\mathcal N\subseteq \mathcal M$ of von Neumann algebras we denote by $\mathcal N'\cap \mathcal M =\{ x\in \emm \,:\, [x, \enn]=0\}$.  We also denote by $\mathscr N_\emm(\enn)=\{ u\in \mathscr U(\emm)\,:\, u\enn u^*=\enn\}$ the normalizing group. We also denote the quasinormalizer of $\enn$ in $\emm$ by $\mathscr {QN}_{\mathcal M}(\mathcal N)$. Recall that $\mathscr {QN}_{\mathcal M}(\mathcal N)$ is the set of all $x\in\mathcal M$ for which there exist $x_1,x_2,...,x_n \in \mathcal M$ such that $\mathcal N x\subseteq \sum_i x_i \mathcal N$ and $x\mathcal N \subseteq \sum_i  \mathcal N x_i$ (see \cite[Definition 4.8]{Po99}).
	\vskip 0.03in
	All von Neumann algebras $\emm$ considered in this document will be tracial, i.e.\ endowed with a unital, faithful, normal linear functional $\tau:M\rightarrow \mathbb C$  satisfying $\tau(xy)=\tau(yx)$ for all $x,y\in \emm$. This induces a norm on $\emm$ by the formula $\|x\|_2=\tau(x^*x)^{1/2}$ for all $x\in \emm$. The $\|\cdot\|_2$-completion of $\emm$ will be denoted by $L^2(\emm)$.  For any von Neumann subalgebra $\mathcal N\subseteq \mathcal M$ we denote by $E_{\mathcal N}:\mathcal M\rightarrow \mathcal N$ the $\tau$-preserving condition expectation onto $\mathcal N$. We denote the orthogonal projection from $L^2(\emm) \rightarrow L^2(\enn)$ by $e_{\enn}$. The Jones' basic construction \cite[Section 3]{Jo83} for $\enn \subseteq \emm$ will be denoted by $\langle \emm, e_{\enn} \rangle$.  
	\vskip 0.05in
	For any group $G$ we denote by $(u_g)_{g\in G} \subset \mathscr U(\ell^2G)$ its left regular representation, i.e.\ $u_g(\delta_h ) = \delta_{gh}$ where $\delta_h:G\rightarrow \mathbb C$ is the Dirac function at $\{h\}$. The weak operatorial closure of the linear span of $\{ u_g\,:\, g\in G \}$ in $\mathscr B(\ell^2 G)$ is called the group von Neumann algebra and will be denoted by $\El(G)$; this is a II$_1$ factor precisely when $G$ has infinite non-trivial conjugacy classes (icc). If $\mathcal M$ is a tracial von Neumann algebra and $G \ca^\sigma \mathcal M$ is a trace preserving action we denote by $\mathcal M \rtimes_\sigma G$ the corresponding cross product von Neumann algebra \cite{MvN37}. For any subset $K\subseteq G$ we denote by $P_{\mathcal M K}$  the orthogonal projection from the Hilbert space $L^2(\mathcal M \rtimes G)$ onto the closed linear span of $\{x u_g \,|\, x\in \mathcal M, g\in K\}$. When $\mathcal M$ is trivial we will denote this simply by $P_K$.  
	\vskip 0.05in
	All groups considered in this article are countable and will be denoted by capital letters $A$, $B$, $G$, $H$, $Q$, $N$ ,$M$,  etc. Given groups $Q$, $N$ and an action $Q\curvearrowright^{\sg} N$ by automorphisms we denote by $N\rtimes_\sigma Q$ the corresponding semidirect product group. For any $n\in N$ we denote by ${\rm Stab}_Q(n)=\{ g\in Q\,:\, \sigma_g(n)=n\}$. Given a group inclusion $H \leqslant G$ sometimes we consider the centralizer $C_G(H)$ and the  virtual centralizer $vC_G(H)=\{g\in G \,:\, |g^{H}|<\infty\} $.

	\subsection{Popa's Intertwining Techniques} Over more than fifteen years ago, Popa  introduced  in \cite [Theorem 2.1 and Corollary 2.3]{Po03} a powerful analytic criterion for identifying intertwiners between arbitrary subalgebras of tracial von Neumann algebras. Now this is known in the literature  as \emph{Popa's intertwining-by-bimodules technique} and has played a key role in the classification of von Neumann algebras program via Popa's deformation/rigidity theory.

	\begin {theorem}\cite{Po03} \label{corner} Let $(\mathcal M,\tau)$ be a separable tracial von Neumann algebra and let $\mathcal P, \mathcal Q\subseteq \mathcal M$ be (not necessarily unital) von Neumann subalgebras. 
	Then the following are equivalent:
	\begin{enumerate}
		\item There exist $ p\in  \mathscr P(\mathcal P), q\in  \mathscr P(\mathcal Q)$, a $\ast$-homomorphism $\theta:p \mathcal P p\rightarrow q\mathcal Q q$  and a partial isometry $0\neq v\in q \mathcal M p$ such that $\theta(x)v=vx$, for all $x\in p \mathcal P p$.
		\item For any group $\mathcal G\subset \mathscr U(\mathcal P)$ such that $\mathcal G''= \mathcal P$ there is no sequence $(u_n)_n\subset \mathcal G$ satisfying $\|E_{ \mathcal Q}(xu_ny)\|_2\rightarrow 0$, for all $x,y\in \mathcal  M$.
		\item There exist finitely many $x_i, y_i \in \mathcal M$ and $C>0$ such that  $\sum_i\|E_{ \mathcal Q}(x_i u y_i)\|^2_2\geq C$ for all $u\in \mathcal U(\mathcal P)$.
	\end{enumerate}
\end{theorem} 
\vskip 0.02in
\noindent If one of the three equivalent conditions from Theorem \ref{corner} holds then we say that \emph{ a corner of $\mathcal P$ embeds into $\mathcal Q$ inside $\mathcal M$}, and write $\mathcal P\prec_{\mathcal M}\mathcal Q$. If we moreover have that $\mathcal P p'\prec_{\mathcal M}\mathcal Q$, for any projection  $0\neq p'\in \mathcal P'\cap 1_{\mathcal P} \mathcal M 1_{\mathcal P}$ (equivalently, for any projection $0\neq p'\in\mathscr Z(\mathcal P'\cap 1_{\mathcal P}  \mathcal M 1_{P})$), then we write $\mathcal P\prec_{\mathcal M}^{s}\mathcal Q$. We refer the readers to the survey papers \cite{Po07, Va10icm, Io18icm} for recent progress in von Neumann algebras using deformation/rigidity theory.
\vskip 0.02in
We also recall the notion of relative amenability introduced by N. Ozawa and S. Popa. Let $(\emm, \tau) $ be a tracial von Neumann algebra. Let $p \in \emm$ be a projection, and let $\mathcal P \subseteq p \emm p$, and $\mathcal Q \subseteq \emm$ be von Neumann subalgebras. Following \cite[Definition 2.2]{OP07}, we say that $\mathcal P$ is amenable relative to $\mathcal Q$ inside $\emm$, if there exists a positive linear functional $\phi: p \langle \emm, e_{\mathcal Q} \rangle p \rightarrow \mc$ such that $\phi|_{p \emm p}= \tau$ and $\phi(xT)=\phi(Tx)$ for all $T \in \mathcal Q$ and all $x \in \mathcal P$. If $\mathcal P$ is amenable relative to $\mathcal Q$ inside $\emm$, we write $\mathcal P \lessdot_{\emm} \mathcal Q$.

\subsection{Class $\mathscr S$ and $\mathscr V$} \label{grnot}
We now briefly recall the classes of groups, Class $\mathscr S$ and $\mathscr V$, that were introduced in \cite{CDHK20}, and triviality of fundamental group for factors arising from these classes of groups was established. We refer the reader to \cite[Section 3 ]{CDHK20} for more details about these groups. We shall follow a bit different notation from that of \cite{CDHK20}, and hence we explain the notation used in this paper below.
\vskip 0.02 in
\textbf{Class $\mathscr V$:} Denote by $\mathbb H$ the division algebra of quaternions and by $\mathbb H_{\mathbb Z}$ its lattice of integer points. Let $n \geq 2$. Recall that $\Sp(n, 1)$ is the rank one connected simple real Lie group defined by 
$$\Sp(n, 1) = \left\{ A \in \GL_{n+ 1}(\mathbb H) \mid A^* J A = J\right\}$$
where $J = \Diag(1, \dots, 1, -1)$.  Since the subgroup $\Sp(n, 1)$ is the set of real points of an algebraic $\mathbb Q$-group, the group of integer points $\Lambda_n = \Sp(n, 1)_{\mathbb Z}$ is a lattice in $\Sp(n, 1)$ by Borel--Harish-Chandra's result \cite{BHC61}. Observe that $\Sp(n, 1)$ acts linearly on $\mathbb H^{n + 1} \cong \mathbb R^{4(n + 1)}$ in such a way that $\Lambda_n$ preserves $(\mathbb H_{\mathbb Z})^{n+1} \cong \mathbb Z^{4(n + 1)}$. For every $n \geq 2$, consider the natural semidirect product $G_n =  \mathbb Z^{4(n + 1)} \rtimes \Lambda_n$. 
\vskip 0.02 in
\noindent Throughout this document we denote by $\mathscr V_1$ the class of groups $G_n=\mathbb Z^{4(n + 1)} \rtimes \Lambda_n$, with $n \geq 2$. We also denote by $\mathscr V_m$ the class of all $m$-fold products of groups in $\mathscr V_1$. Finally we let $\mathscr V= \cup_m \mathscr V_m$, the class of all finite direct products of groups in $\mathscr V_1$.
\vskip 0.02 in
It's easy to see that $\El(\G)$ is prime if $\G \in \mathscr V_1$. Nevertheless, we include a short sketch of proof using \cite{PV12} for readers' convenience.
\begin{lemma}
	Let $\G \in \mathscr V_1$. Then $\El(\G)$ is prime.
\end{lemma}
\begin{proof}
	Suppose $\El(\G)=\mathcal A \rtimes H=\pee_1 \tp \pee_2$, where $\pee_1$ and $\pee_2$ are $\rm II_1$ factors, $\mathcal A= \El(\mathbb Z^{4(n + 1)})$, and $H= \Lambda_n$ is hyperbolic. Using \cite[Theorem 1.1]{PV12}, we get that either $\pee_1 \prec_{\El(\G)} \mathcal A$, or $\pee_2 \prec_{\El(\G)} \mathcal A$, which is impossible, as both $\pee_1$ and $\pee_2$ are property (T) factors.
\end{proof}
\textbf{Class $\mathscr S$:} Consider any product group $Q= Q_1\times Q_2$, where $Q_i$ are any nontrivial, bi-exact, weakly amenable, property (T), residually finite, torsion free, icc groups. Then for every $i=1,2$  consider a Rips construction  $G_i = N_i \rtimes_{\rho_i} Q\in \mathcal Rips(Q)$, let $N=N_1\times N_2$  and denote by $G= N\rtimes_\sigma Q$ the canonical semidirect product which arises from the diagonal action  $\sigma=\rho_1\times \rho_2: Q\rar {\rm Aut}(N)$, i.e. $\sigma_g (n_1,n_2)=( (\rho_1)_g(n_1), (\rho_2)_g(n_2))$ for all $(n_1,n_2)\in N$. Throughout this article the category of all these  semidirect products $G$ will denoted by  Class $\mathscr S$.
\vskip 0.02 in
We end this section by showing that $\El(\G)$ is prime if $\G \in \mathscr S$. Note that this result can be deduced from \cite{CKP14}. We shall provide a very short proof of this fact, using techniques from \cite{CDHK20} for readers' convenience.
\begin{lemma}
	Let $\G \in \mathscr S$. Then $\El(\G)$ is prime.
\end{lemma}
\begin{proof}
		Suppose $\El(\G)=\mathcal L((N_1 \times N_2) \rtimes Q)=\pee_1 \tp \pee_2$. 
			 Since $\pee_1$ and $\pee_2$ are commuting property (T) subfactors of $\emm$ such that $(\pee_1 \vee \pee_2)' \cap \emm= \mc$, using \cite[Theorem 4.3]{CDHK20}, we get that\begin{enumerate}
			\item[a)] Either $\pee_1 \vee \pee_2 \prec_{\emm}^s  \El(N_1 \times N_2)$ , or 
			\item [b)] $\pee_1 \vee \pee_2 \prec_{\emm}^s  \El(Q)$.
		\end{enumerate}
			Note that as $\pee_1 \vee \pee_2= \El(\G)$, both cases (a) and (b) above lead to a contradiction. Hence $\El(\G)$ is prime, if $\G \in \mathscr S$.  
		
\end{proof}
%%%%%%%%%%%%%%%%%%%%%%%%%%%%%%%%%%%%%%%%%%%%%%%%%%%%%%%%%%%%%%%%%%%%%%%%%%%%%%%%%%%%%%%%%%%%%%%%%%%%%%%%%%%%%%%%%%%%%%%%%%%%%%%%%%%%%%%%%%%%%%%%%%%%%%%%%%%%%%%%%%%%%%%%%%%%%%%%%%%%%%%%%%%%%%%%%%%%%%%%%%%%%%%%%%%%%%%%%%%%%%%%%%%%%%%%%%%%%%%%%%%%%%%%%%%%%%%%%%%%%%%%%%%%%%%%%%%%%%%%%%%%%%%%%%%%%%%%%%%%%%%%%%%%%%%%%%%%%%%%%%%%%%%%%%%%%%%%%%%%%%%%%%%%%%%%%%%%%%%%%%%%%%%%%%%%%%%%%%%%%%%%%%%%%%%%%%%%%%%%%%%%%%%%%%%%%%%%%%%%%%%%%%%%%%%%%%%%%%%%

\section{Unique Prime Factorization and Fundamental Group Calculations}
In this section we provide new examples of property (T) group factors with trivial fundamental group. Our examples arise as direct product groups $\G= \G_1 \times \G_2$, where $\G_i \in \mathscr S \cup \mathscr V$. Class $\mathscr V$ was introduced in \cite{CDHK20}, while Class $\mathscr S$ was introduced in \cite{CDK19}. Throughout this section we shall use the notations as in Section ~\ref{grnot}. Note that the case when both $\G_1$ and $\G_2$ belong to Class $\mathscr V$ was considered already in \cite[Theorem 5.2]{CDHK20}, and hence we shall only focus on the remaining cases. The result on the calculation of fundamental group shall be obtained as a corollary to unique prime factorization results for $\El(\G)$, in conjunction with results from \cite{CDHK20}. 
\vskip 0.02 in
We begin by establishing UPF for product of groups in $\mathscr S$ and $\mathscr V$.
\begin{proposition} \label{upf2ind}
	Let $\G \in \mathscr S$ and $\La \in \mathscr V_1$ be as in Section ~\ref{grnot}. Then $\El(\G \times \La)$ admits a Unique Prime Factorization.
\end{proposition}

\begin{proof}
	Let $\La = A \rtimes H$, where $A$ is amenable, and $H$ is hyperbolic. Suppose $\emm:= \El(\G \times \La)= \pee_1 \tp \pee_2$. Note that $\emm= \El(\G \times A) \rtimes H$, where $H \ca \El(\G)$ is the trivial action. Let $\bee \subseteq \pee_1$ be a diffuse, amenable subalgebra. Using \cite{PV12}, we have that either \begin{enumerate}
		\item[$i)$] $\bee \prec_{\emm} \El(\G \times A)$, or
		\item[$ii)$] $\pee_2 \subseteq \mathcal N_{\emm}(\bee)'' \lessdot_{\emm} \El(\G \times A)$.
	\end{enumerate}
If case $i)$ above holds for all $\bee \subseteq \pee_1$ diffuse amenable, then using \cite[Corollary F.14]{BO08}, we get that $\pee_1 \prec_{\emm} \El(\G \times A)$. If case $ii)$ above holds for some $\bee \subseteq \pee_1$, then $\pee_2 \lessdot_{\emm} \El(\G \times A)$. However, this implies that $\pee_2 \prec_{\emm} \El(\G \times A)$, as $\pee_2$ has property (T). Without loss of generality, assume that $\pee_1 \prec_{\emm} \El(\G \times A)$. 
\vskip 0.02 in
Thus, there exists nonzero projections $p \in P_1$, $q \in \El(\G \times A)$, a nonzero partial isometry $v \in q \emm p$, and a unital $\ast$-homomorphism 
\begin{equation} \label{int1}
\phi: p \pee_1 p \rightarrow \mr:= \phi(p \pee_1 p) \subseteq q \El(\G \times A) \text{ such that } \phi(x)v=vx \text{ for all }x \in p \pee_1 p.
\end{equation}
  Note that $v^*v \in (p \pee_1 p)' \cap p \emm p= P_2 p$, and $vv^* \in \mr' \cap q \emm q$.
  Moreover, we may arrange so that $support(E_{\El(\G \times A)}(vv^*))=q$. Since $\mr$ has property (T) and $\El(A)$ is amenable, we must have $\mr \prec_{\El(\G \times A)} \El(\G)$. So we can find nonzero projections $r \in \mr$, $s \in \El(\G)$, a nonzero partial isometry $w \in s \El(\G \times A)r$ and a $\ast$-homomorphism
  \begin{equation} \label{int2}
  \psi: r \mr r \rightarrow  s \El(\G )s \text{ such that } \psi(x)w=wx \text{ for all }x \in r \mr r.
  \end{equation} 
  Let $t = \phi^{-1}(r) \in p \pee_1 p$. Let $y \in t(p \pee_1 p)t = t \pee_1 t$. Then we have:
  \begin{equation} \label{intcom}
  \psi(\phi(y))wv= w \phi(y)v = wv y.
  \end{equation}
  We briefly argue that $wv \neq 0$. Assume that $wv=0$. Then, $wvv^*=0$. As $w \in \El(\G \times A)$, we get $w E_{\El(\G \times A)}(vv^*)=0$. This implies that $w {\rm {support}}E_{\El(\G \times A)}(vv^*)=0 \Rightarrow wq=0$. However, this implies that $w=0$, which is a contradiction.  
  \vskip 0.02in
  Hence, by equation ~\ref{intcom} we get that $\pee_1 \prec_{\emm} \El(\G)$. By \cite[Proposition 12]{OP03} we get that there exists $s>0$, and $u \in \euu(\emm)$ such that $\pee_1 \subseteq u \El(\G)^s u^*$. By \cite{Ge}, we get that there exists $\mathcal Q \subseteq \pee_2$ von Neumann subalgebra such that $u \El(\G)^s u^*= \pee_1 \tp Q$. As $\El(\G)$ is a prime factor, we must have that $\mathcal Q$ is a finite dimensional factor. Thus, there exists $t>0$ such that $u \El(\G)^t u^*= \pee_1$. Hence $\pee_2= u \El(\La)^{1/s}u^*$. This proves the theorem.
\end{proof}
We are now ready to prove one of the main results of this section.
\begin{theorem} \label{upf2}
		Let $\G \in \mathscr S$ and $\La \in \mathscr V$ be as in Section ~\ref{grnot}. Then $\emm =\El(\G \times \La)$ admits a Unique Prime Factorization.
\end{theorem}

\begin{proof}
	Since $\La \in \mathscr V$, $\La= \La_1 \times \cdots \La_n$, where $\La_i = A_i \rtimes H_i \in \mathscr V_1$ for all $i$. Assume $\emm= \pee_1 \tp \pee_2$.
	
	So, $\pee_1 \tp \pee_2= \El(\G \times \La_1 \times \cdots \La_{n-1} \times A_n) \rtimes H_n$. As $H_n$ is hyperbolic, and $\pee_i$'s have property (T), using \cite[Theorem 1.4]{PV12} we get that $\pee_i \prec_{\emm} \El(\G \times \La_1 \times \cdots \La_{n-1} \times A_n)$ for some $i \in \{1,2\}$. Without loss of generality, assume that $\pee_1 \prec_{\emm} \El(\G \times \La_1 \times \cdots \La_{n-1}) \tp \El(A_n)$. Since $\pee_1$ has property (T), and $\El(A_n)$ is amenable, arguing exactly as in the proof of Proposition ~\ref{upf2ind}, we get that $\pee_1 \prec_{\emm} \El(\G \times \La_1 \times \cdots \La_{n-1})$. Taking relative commutants, and using \cite[Lemma 3.5]{Va08} we get that $\El(\La_n) \prec_{\emm} \pee_2$. Using \cite[Proposition 12]{OP03}, there exists $u_n \in \euu(\emm)$, $t_n>0$, such that:
	\begin{equation*}
	\El(\La_n) \subseteq u_n \pee_2^{t_n} u_n^* \subseteq \El(\G \times \La_1 \cdots \times \La_{n-1}) \tp \El(\La_n).
	\end{equation*}
	Using \cite{Ge}, we get that $u_n \pee_2^{t_n} u_n^*= \mathcal Q \tp \El(\G \times \La_1 \cdots \times \La_{n-1})$, where $\mathcal Q \subseteq \El(\G \times \La_1 \cdots \times \La_{n-1})$. Moreover, by taking relative commutants, we get that $u_n \pee_1^{1/t_n}u_n^*= \mathcal Q' \cap \El(\G \times \La_1 \cdots \times \La_{n-1})$. Since $\pee_1 \vee \pee_2 = \emm$, we get that $u_n \pee_1^{1/t_n}u_n^* \vee Q = \El(\G \times \La_1 \times \cdots \times \La_{n-1})$. The result now follows by induction, and Proposition ~\ref{upf2ind}.
\end{proof}

\begin{corollary} \label{fundpdtclassSV}
	Let $\G \in \mathscr S$ and $\La \in \mathscr V$ be as above. Then $\mathcal F(\El(\G \times \La))=1$.
\end{corollary}

\begin{proof}
	Let $t>0$ be such that $\El(\G \times \La) \cong \El(\G)^t \tp \El(\La)$. Using Theorem~\ref{upf2}, we get that there exists $s>0$ such that:
	\begin{enumerate}
		\item [a)] Either $\El(\G) \cong \El(\G)^{ts} $ and $\El(\La) \cong \El(\La)^{1/s}$ or,
		\item [b)] $\El(\G) \cong \El(\La)^{s}$ and $\El(\La) \cong \El(\G)^{t/s}$.
	\end{enumerate}

Note that case b) above can't hold, as $\El(\La)$ has a Cartan subalgebra by \cite{CDHK20}, while $\El(\G)$ does not. Therefore, case a) above holds. By \cite{CDHK20}, we further get that $st=1$ and $s=1$, which forces $t=1$.
\end{proof}
We now establish UPF for products of groups in $\mathscr S$. This is the main technical result of this section.
\begin{theorem} \label{upfclassS}
	Let $\G_1, \G_2 \in \mathscr S$. Then $\emm:= \El(\G_1 \times \G_2)$ admits a Unique Prime Factorization.
\end{theorem}

\begin{proof}
	We let $\G_1=(N_1 \times N_2) \rtimes Q$, and $\G_2=(M_1 \times M_2) \rtimes P$. Suppose there exits $\rm II_1$ factors $\ms_1$ and $\ms_2$ such that $\emm=\ms_1 \tp \ms_2$. To establish the desired result, it suffices to show that $\ms_i \prec_{\emm} \El(G_1)$ for $i=1$ or $i=2$.\\	
	Let $G_1= N_1 \rtimes Q$, $G_2=N_2 \rtimes Q$, $G_3=M_1 \rtimes P$ and $G_4=M_2 \rtimes P$. We also note that $G_1$ is hyperbolic relative to $H_1=Q$, $G_2$ is hyperbolic relative to $H_2=Q$, $G_3$ is hyperbolic relative to $H_3=P$ and $G_4$ is hyperbolic relative to $H_4=P$
	 We denote $\El(G_1 \times G_2\times G_3 \times G_4)$by $\tilde{\emm}$. Note that $\emm \subseteq \tilde{\emm}$, and $\emm' \cap \tilde{\emm}=\mc$. \\	 
	Note that $\ms_1$ and $\ms_2$ are commuting property (T) subfactors of $\tilde{\emm}$. Using \cite[Theorem 5.3]{CDK19}, for every $k \in \{1,2,3,4\}$ there exists $i \in \{1,2\}$ such that one of the following holds:
	 \begin{enumerate}
	 	\item [a)] Either  $\ms_i \prec_{\tilde{\emm}} \El(\hat{G_k})$or,
	 	\item [b)] $\ms_1 \vee \ms_2 \prec_{\tilde{\emm}} \El(\hat{G_k} \times H_k)$.
	 \end{enumerate}
 We briefly argue that Case b) above doesn't hold for any $k \in \{1,2,3,4\}$. For ease of presentation, we shall only consider $k=2$ ( For $k=1,3,4$ the proof is entirely analogous).\\ 
 Assume that $\emm= \ms_1 \vee \ms_2 \prec_{\tilde{\emm}} \El(\hat{G_2} \times H_2)$. So, $\El(\G_1 \times \G_2)\prec_{\tilde{\emm}} \El(G_1 \times H_2 \times G_3 \times G_4)$. This implies that $\El(\G_1) \prec_{\El(G_1 \times G_2)} \El(G_1 \times H_2)$. By \cite[Lemma 2.3]{CDK19}, there exists $h \in G_1 \times G_2$ such that  $\El((N_1 \times N_2) \rtimes diag(Q)) \prec_{\El(G_1 \times G_2)} \El(((N_1 \times N_2) \rtimes Q)) \cap h((N_1 \rtimes Q) \times Q)h^{-1})$. Note that we may assume $h \in N_2$. So, $(N_1 \times N_2) \rtimes Q \cap h((N_1 \rtimes Q) \times Q)h^{-1})= h((N_1 \times 1) \rtimes diag(Q))h^{-1}$. Hence we get that $\El((N_1 \times N_2) \rtimes diag(Q)) \prec_{\El(G_1 \times G_2)} \El((N_1 \times 1) \rtimes diag(Q))$. Applying \cite[Lemma 2.5]{CDK19} we get that $\El((N_1 \times N_2) \rtimes diag(Q)) \prec_{(N_1 \times N_2) \rtimes diag(Q)} \El((N_1 \times 1) \rtimes Q)$. As $[(N_1 \times N_2) \rtimes diag(Q): (N_1 \times 1) \rtimes diag(Q)]= \infty$, this is impossible (for instance, by using \cite[Proposition 2.3]{CD18}).
\vskip 0.02 in
Thus, we have that for each $k \in \{1,2,3,4 \}$, there exists $i \in \{1,2\}$ 
such that $\ms_i \prec_{\tilde{\emm}} \El(\hat{G_k})$. Applying this for $k=4$, we get that there exists $i \in \{1,2\}$ such that $\ms_i \prec_{\tilde{\emm}} \El(G_1 \times G_2 \times G_3)$. 
As $\ms_i \subseteq \El(\G_1 \times \G_2)$, % and $\G_1 \times \G_2 \subseteq G_1 \times \cdots \times G_4$
\cite[Lemma 2.3]{CDK19} yields that $\ms_i \prec_{\tilde{\emm}} \El((\G_1 \times \G_2) \cap g(G_1 \times G_2 \times G_3)g^{-1})$ for some $g \in G_1 \times G_2 \times G_3 \times G_4$. Note that $(\G_1 \times \G_2) \cap g(G_1 \times G_2 \times G_3)g^{-1}= \G_1 \times (\G_2 \cap G_3)=\G_1 \times (M_1 \times 1)$. So, $\ms_i \prec_{\tilde{\emm}} \El(\G_1) \tp \El(M_1)$. \\
Arguing exactly as above for $k=3$, we get that there exits $j \in \{1,2\}$ such that $\ms_j \prec_{\tilde{\emm}} \El(\G_1) \tp \El(M_2)$. 
\vskip 0.02 in
We claim that $i =j$. Otherwise, we would have $\ms_1 \prec_{\tilde{\emm}} \El(\G_1) \tp \El(M_2)$ and $\ms_2 \prec_{\tilde{\emm}} \El(\G_1) \tp \El(M_1)$ (or vice versa). Either way, we have that $\ms_i \prec_{\tilde{\emm}} \El(G_1 \times G_2) \tp \El(M_1 \times M_2)$ for $i=1,2$. 
Note that $\mathcal N_{\tilde{\emm}}(\ms_i)' \cap \tilde{\emm} \subseteq (\ms_1 \vee \ms_2)' \cap \tilde{\emm} = \emm' \cap \tilde{\emm} =\mc$. As $\ms_1$ and $\ms_2$ are commuting subfactors of $\tilde{\emm}$, and as $\El(G_1 \times G_2) \tp \El(M_1 \times M_2)$ is regular in $\tilde{\emm}$, by \cite[Lemma 2.6]{Is20}, we get that $\emm= \ms_1 \vee \ms_2 \prec_{\tilde{\emm}} \El(G_1 \times G_2) \tp \El(M_1 \times M_2)$. Note that by \cite[Lemma 2.6]{CDK19} we actually have that $\emm=\El(\G_1 \times \G_2) \prec{\El(\G_1)\tp \El(G_3 \times G_2)}\El(\G_1) \tp \El(M_1 \times M_2)$, which in turn implies that $\El(\G_2) \prec_{\El(G_3 \times G_4)} \El(M_1 \times M_2)$. Using \cite[Lemma 2.4]{CDK19}, we get that $\El(\G_2) \prec_{\El(\G_2)} \El(M_1 \times M_2)$, which is impossible (for instance by \cite[Proposition 2.3]{CD18}).
\vskip 0.02 in
Thus we get that there exists $i$ such that $\ms_i \prec_{\tilde{\emm}} \El(\G_1 ) \tp \El(M_1)$ and $\ms_i \prec_{\tilde{\emm}} \El(\G_1 ) \tp \El(M_2)$. Note that, as $\mathcal N_{\tilde{\emm}}(\ms_i)' \cap \tilde{\emm} =\mc$, using \cite[Lemma 2.4(2)]{DHI16}, we actually have that  $\ms_i \prec_{\tilde{\emm}}^s \El(\G_1 ) \tp \El(M_1)$ and $\ms_i \prec_{\tilde{\emm}}^s \El(\G_1 ) \tp \El(M_2)$. Now, using \cite[Lemma 2.8(2)]{DHI16},  we get that $\ms_i \prec_{\tilde{\emm}}^s (\El(\G_1 ) \tp \El(M_1)) \cap (\El(\G_1 ) \tp \El(M_2)) = \El(\G_1)$. Now, by \cite[Lemma 2.6]{CDK19} we in fact have $\ms_i \prec_{\emm} \El(\G_1)$. Applying \cite[Proposition 12]{OP03} we get the desired result.
\end{proof}
\begin{corollary} \label{fundpdtclassS}
	Let $\G_1, \G_2 \in \mathscr S$. Then $ \mathcal F(\El(\G_1 \times \G_2))= \{1\}$.
\end{corollary}
\begin{proof}
	Suppose $t \in \mathcal F(\El(\G_1 \times \G_2))$. Then $\El(\G_1)\ tp \El(G_2) \cong \El(\G_1)^t \tp \El(\G_2) $. Then by Theorem~\ref{upfclassS}, there exists $s>0$ such that
	\begin{enumerate}
		\item[a)] Either $\El(\G_1) \cong \El(\G_1)^{st}$ and $\El(\G_2) \cong \El(\G_2)^{1/s}$; or
		\item [b)] $\El(\G_1) \cong \El(\G_2)^s$ and $\El(\G_2) \cong \El(G_1)^{t/s}$.
	\end{enumerate}  
Since $\mathcal F(\El(G_1))= \mathcal F(\El(G_2)) =\{1\}$, Case a) above immediately yields that $st=1$ and $1/s=1$. This implies $t=1$. If Case b) above holds, then using \cite[Theorem 4.6]{CDHK20}, we get that $s=1$, $t/s=1$, which implies that $t=1$.
\end{proof}
\vskip 0.02 in
\textbf{Final Remarks:} As mentioned in the introduction, Theorem ~\ref{upf2}, and Theorem ~\ref{upfclassS} provide further evidence towards Popa's conjecture on prime decomposition of property (T) factors (second last problem on page 9 in \cite{Po13}; see also Problem S.4 in Peterson's problem list \cite{PProb}), while Corollary ~\ref{fundpdtclassSV} and Corollary ~\ref{fundpdtclassS} provide further evidence towards Popa's conjecture on triviality of Fundamental group of property (T) group factors (last problem on page 9 in \cite{Po13}).

\section{Some More UPF Results}
In this section, we prove two additional unique prime factorization results. Our first result (Theorem ~\ref{resfinupf}) establishes a unique prime factorization result for direct products of i.c.c., property (T) groups, that are hyperbolic relative to infinite residually finite subgroups. Our second result (Theorem ~\ref{upf}) establishes unique prime factorization of tensor product of a prime full factor with a group von Neumann algebra of an icc,hyperbolic group.

The following proposition establishes the base case of induction used in Theorem~\ref{resfinupf}.
\begin{proposition} \label{basecase}
	For $i=1,2$ let $H_i \leq G_i$ be inclusions of infinite groups such that $H_i$ is residually finite. Assume that $[G_2:H_2]= \infty$. Further assume that each $G_i$ is icc, property (T) group, and that $G_i$ is hyperbolic relative to $H_i$ for all $i$. Then $\enn:= \El(G_1 \times G_2)$ admits a UPF.
\end{proposition}

\begin{proof}
	Suppose $\enn:=\El(G_1 \times G_2)= \pee_1 \tp \pee_2$. By \cite[Theorem 5.3]{CDK19}, either $\pee_1 \vee \pee_2 \prec_{\enn} \El(G_1 \times H_2)$, or $\pee_i \prec_{\enn} \El(G_1)$. Note that $\pee_1 \vee \pee_2 \prec_{\enn} \El(G_1 \times H_2)$ implies that $H_2$ has finite index in $G_2$, which contradicts our hypothesis. Without loss of generality, assume that $\pee_1 \prec_{\enn} \El(G_1)$. By \cite[Proposition 12]{OP03}, we have that there exists $u \in \euu(\enn)$, and $s>0$ such that $\pee_1 \subseteq u \El(G_1)^s u^* \subseteq \pee_1 \tp \pee_2$. Using \cite{Ge}, we conclude that $u \El(G_1)^s u^*= \pee_1 \tp \mathcal B$ for some $\bee \subseteq \pee_2$. As $\El(G_1)$ is a prime factor by \cite[Theorem A]{CKP14}, we conclude that $\bee$ is a finite dimensional factor. Hence, $u \El(G_1)^t u^*= \pee_1$ for some $t>0$. Taking relative commutants, we get that $\pee_2= u \El(G_2)^{1/t}u^*$, thereby proving the proposition.
\end{proof}
\begin{theorem} \label{resfinupf}
For $i=1,\cdots,n$ let $H_i \leq G_i$ be inclusions of infinite groups such that $H_i$ is residually finite, and of infinite index in $G_i$ for all $i$. Further assume that each $G_i$ is icc, property (T) group, and that $G_i$ is hyperbolic relative to $H_i$ for all $i$. Then $\emm:= \El(G_1 \times \cdots \times G_n)$ admits a UPF. 
\end{theorem}

\begin{proof}
	Let $\emm= \pee_1 \tp \pee_2$. Note that $\pee_i$ is a property (T) factor for $i=1,2$. Using \cite[Theorem 5.3]{CDK19}, we get that:
	\begin{enumerate}
		\item [a)] Either $\pee_i \prec_{\emm} \El(G_1 \times \cdots G_{n-1})$ for some $i\in \{1,2\}$, or
		\item [b)] $\pee_1 \vee \pee_2 \prec_{\emm} \El(G_1 \times \cdots G_{n-1} \times H_n)$.
	\end{enumerate}
Assume case b) above holds. Then $H_n$ has finite index in $G_n$, which is a contradiction. Thus case a) holds. Without loss of generality, assume that $\pee_1 \prec_{\emm} \El(G_1 \times \cdots G_{n-1})$. By \cite[Lemma 3.5]{Va08} we get that $\El(G_n) \prec_{\emm} P_2$. By \cite[Proposition 12]{OP03}, there exists $u_n \in \euu(\emm)$, $t_n >0$ such that $\El(G_n) \subseteq u_n \pee_2^{t_n} u_n^* \subseteq \El(G_1 \times \cdots \times G_{n-1}) \tp \El(G_n)$. By \cite{Ge}, we get that $u_n \pee_2^{t_n}u_n^* = \mathcal Q \tp \El(G_n)$, where $\mathcal Q \subseteq \El(G_1 \times \cdots \times G_{n-1})$. 
\vskip 0.02in
Since $\mathcal Q' \cap \El(G_1 \times \cdots \times G_{n-1})=(u_n \pee_2^{t_n}u_n^*)' \cap \emm = u_n\pee_1^{1/t_n}u_n^*$, we get that $\mathcal Q$ and $u_n\pee_1^{1/t_n}u_n^*$ are commuting property (T) subfactors of $\El(G_1 \times \cdots \times G_{n-1})$. Also, $\pee_1 \vee \pee_2 =\emm$ clearly implies that $\mathcal Q \vee u_n\pee_1^{1/t_n}u_n^* = \El(G_1 \times \cdots \times G_{n-1})$. The result now follows by induction on $n$, and Proposition ~\ref{basecase}.
\end{proof}

We end this section with a nice application of \cite[Theorem 1.4]{PV12} to establish UPF for tensor product of prime full factors with a group factor arising from an icc,hyperbolic group.
\begin{theorem} \label{upf}
	Let $H$ be an icc, hyperbolic group, and let $\enn$ be a full prime factor. Let $\emm=\El(H) \tp \enn$. Then $\emm$ satisfies the UPF property; i.e. whenever $\emm = \pee_1 \tp \pee_2$, there exists $u \in \euu(\emm)$ and $t>0$ such that $\El(H)=u\pee_i^tu^*$ and $\enn= u\pee_{i+1}^{1/t}u^*$. 
\end{theorem}

\begin{proof}
	Let $\emm= \pee_1 \tp \pee_2$. As $H$ is nonamenable, hyperbolic, $\El(H)$ is a full factor by \cite[Cor B]{CSU13}. Thus, $\emm$ is a full factor by \cite[Cor 2.3]{Co76}. Hence, $\pee_1$ and $\pee_2$ are also full factors, by \cite[Cor 2.3]{Co76}.
	
	Let $\Aa \subseteq \pee_1$ be amenable. Then, using \cite[Theorem 1.4]{PV12}, either 
	\begin{enumerate}
		\item [a)] $\Aa \prec_{\emm} \enn$, or
		
		\item [b)] $\mathcal N_{\emm}(\Aa)'' \lessdot_{\emm} \enn$.
	\end{enumerate}
	
	Assume Case b) above holds for some amenable $\Aa$. As $\pee_2 \subseteq \mathcal N_{\emm}(\Aa)''$, we deduce that $\pee_2 \lessdot_{\emm} \enn$. Since $\pee_2$ is full, using \cite[Lemma 5.2]{IM19}, we get that $\pee_2 \prec_{\emm} \enn$.
	
	Assume Case a) above holds for every amenable $\Aa \subseteq \pee_1$. Then using \cite[Lemma F.14]{BO08}, we get that $\pee_1 \prec_{\emm} \enn$. 
	
	Without loss of generality, assume that $\pee_1 \prec_{\emm} \enn$. Note that $\pee_1' \cap \emm= \pee_2$ is a factor. Hence by \cite[Proposition 12]{OP03}, we get that there exists $s>0$ and a unitary $v \in \euu(\emm)$ such that $\pee_1 \subseteq v \enn^{s} v^*$. Writing $\emm=\pee_1 \tp \pee_2$, using \cite[Theorem 3.1]{Ge} (see also \cite[Cor 3.3]{CD19}), we get that there exists $\mathcal Q \subseteq \pee_2$ such that $v \enn^{s} v^* = \pee_1 \tp \mathcal Q$. As $\enn$ is a prime factor, so is $v \enn^{s} v^*$. Hence $\mathcal Q$ must be a finite dimensional factor, and thus is isomorphic to $\mathbb M_n(\mc)$ for some $n \in \mathbb N$. Altogether, this yields \begin{equation} \label{tensordecom}
	v \enn v^*= \pee_1^t \text{ for some } t>0.
	\end{equation} 
	Hence, using equation~\ref{tensordecom}, we get that $v \El(H)v^{\ast} = (v \enn v^*)' \cap \emm= (\pee_1^t) ' \cap \emm= \pee_2^{1/t}$. Taking $u=v^*$ finishes the proof.
\end{proof}

%%%%%%%%%%%%%%%%%%%%%%%%%%%%%%%%%%%%%%%%%%%%%%%%%%%%%%%%%%%%%%%%%%%%%%%%%%%%%%%%%%%%%%%%%%%%%%%%%%%%%%%%%%%%%%%%%%%%%%%%%%%%%%%%%%%%%%%%%%%%%%%%%%%%%%%%%%%%%%%%%%%%%%%%%%%

\section*{Acknowledgments}
The author is indebted to Prof.\ Ionut Chifan and Prof.\ Jesse Peterson for their invaluable comments and suggestions regarding the contents of this paper, and for their constant encouragement. The author would also like to express his gratitude towards Dr.\ Krishnendu Khan for many stimulating conversations regarding the contents of this paper.\par
This paper is dedicated to the loving memory of Prof.\ Vaughan Jones.

\noindent
%\textsc{Department of Mathematics, The University of Iowa, 14 MacLean Hall, Iowa City, IA 52242, U.S.A.}\\
%\email {ionut-chifan@uiowa.edu} \\
\email{sayan.das@ucr.edu}\\
\textsc{Department of Mathematics, , University of California Riverside, 900 University Ave., Riverside, CA 92521, USA}

%\email{krishnendu.khan@vanderbilt.edu}

\end{document}